\documentclass[12pt]{article}
\usepackage{tikz}
\usepackage{amsfonts}
\usepackage{amsmath}
\usepackage{amsthm}
\usepackage{epic}
\usepackage{eepic}
\usepackage{color}
\usepackage{pst-node}

\newtheorem{prop}{Proposition}

\theoremstyle{remark}

\begin{document}

\title{On domination perfect graphs}
\date{}
\maketitle

\vspace{-20mm}
\begin{center}
\textsc{Jerzy Topp and Pawe\l{}~\.Zyli\'nski}\\[1mm]
University of Gda\'{n}sk, 80-308 Gda\'nsk, Poland\\
{\small \texttt{\{j.topp,zylinski\}@inf.ug.edu.pl}}
\end{center}

\begin{abstract} Let $\gamma(G)$ and $\beta(G)$ denote the domination number and the covering number of a graph $G$, respectively. A connected non-trivial graph $G$ is said to be $\gamma\beta$-{perfect} if $\gamma(H)=\beta(H)$ for every non-trivial induced connected subgraph $H$ of $G$. In this note we present an elementary proof of a characterization of the $\gamma\beta$-perfect graphs. \end{abstract}


\bigskip
\noindent In this note, we follow the notation of~\cite{CLZ15}.
In particular, a subset $D \subseteq V_G$ is a~{\em dominating set} of a~graph $G=(V_G,E_G)$ if each vertex belonging to the set $V_G -D$ has a~neighbor in $D$.
The cardinality of a~minimum dominating set of $G$ is called the~{\em domination number of $G$} and is denoted by $\gamma(G)$. A subset $C \subseteq V_G$ is a~{\em vertex cover} of $G$ if each edge of $G$ has an end-vertex in $C$. (Note that in \cite{AJBT13} a vertex cover is called a~{\em transversal\/} of $G$.) The cardinality of a~minimum vertex cover of $G$ is called the~{\em covering number} of $G$ and is denoted by $\beta(G)$. A connected non-trivial graph $G$ is said to be $\gamma\beta$-{\em perfect\/} if $\gamma(H)=\beta(H)$ for every non-trivial induced connected subgraph $H$ of $G$. Such graphs have been studied in \cite{AJBT13} and ~\cite{DLSZ16}. In this note we compose Theorem 3.9 in \cite{AJBT13} with Theorem 9 in \cite{DLSZ16} and present an elementary proof of the unified result.

We start with two assertions, then give a characterization of the $\gamma\beta$-perfect graphs.

\begin{prop}
\label{Proposition 1} {\rm\cite{DLSZ16}} Every non-trivial tree of diameter at most four and every non-trivial connected subgraph of $K_{2,n}$ is a $\gamma\beta$-perfect graph, while no one of the graphs $C_3$, $C_5$ and $P_6$ is a $\gamma\beta$-perfect graph. \end{prop}

\begin{prop}
\label{Proposition 2}  If $F$ is a connected spanning subgraph of a graph $H$ of order at least three and $\gamma(F)<\beta(F)$, then $\gamma(H)<\beta(H)$ and, therefore, $H$ is not a $\gamma\beta$-perfect graph.
\end{prop}

\begin{proof} Since a dominating set of $F$ is a dominating set of $H$, we have $\gamma(H)\le \gamma(F)$. Similarly, $\beta(F)\le \beta(H)$, since a vertex cover of $H$ is a vertex cover of $F$. Consequently, $\gamma(H)\le \gamma(F)< \beta(F)\le \beta(H)$ and $H$ is not a $\gamma\beta$-perfect graph. \end{proof}

\newpage
\noindent{\bf Theorem.} {\em The following statements are equivalent for a non-trivial connected graph~$G$:
\begin{itemize}
\item[$(1)$] $G$ is a tree of diameter at most four or $G$ is a connected subgraph of $K_{2,n}$.
    \item[$(2)$] $G$ is a $\gamma\beta$-perfect graph.
    \item[$(3)$] $G\not=C_5$ and neither $C_3$ nor $P_6$ is a subgraph of $G$.
\end{itemize} }

\begin{proof} The implication $(1)\Rightarrow (2)$ is obvious from Proposition \ref{Proposition 1}. Assume that $G$ is a~$\gamma\beta$-per\-fect graph. Then, by Proposition~\ref{Proposition 1}, no one of the graphs $C_3$, $C_5$ and $P_6$ is an induced subgraph of $G$. Consequently, $G\not= C_5$ and $C_3$ is not a~subgraph of~$G$. We claim that also $P_6$ is not a subgraph of~$G$.
%
%
Otherwise $P_6$ is a spanning subgraph of some $6$-vertex induced subgraph $H$ of in $G$. Then, since $\gamma(P_6)<\beta(P_6)$, we have  $\gamma(H)<\beta(H)$ (by Proposition~\ref{Proposition 2}), which contradicts the premise that $G$ is $\gamma\beta$-perfect. This proves the implication $(2)\Rightarrow (3)$.
To prove $(3)\Rightarrow (1)$, assume that $G\not= C_5$ and neither $C_3$ nor $P_6$ is a subgraph of $G$.  If $G$ is a tree, then, since $P_6$ is not a~subgraph of $G$, $G$ is of diameter at most 4. Thus assume that $G$ has a cycle, say $C$. Since $G\not=C_5$, the absence of $C_3$ and $P_6$ in $G$ guarantees that $C$ is a~chordless 4-cycle. If $G=C$, then $G=K_{2,2}$. Thus assume that the cycle $C$ is a proper subgraph of $G$. Let $v_1, v_2, v_3, v_4$ be the consecutive vertices of $C$. We may assume without loss of generality that $d_G(v_1)>2$. This time from the absence of $C_3$ and $P_6$ in $G$ it follows that $d_G(v_2)=d_G(v_4)=2$. Now, since $G$ is connected and $P_6$ is not a subgraph of $G$, $N_G(v)\subseteq \{v_1, v_3\}$ for every vertex $v$ belonging to $V_G-\{v_1, v_2, v_3, v_4\}$. Consequently, $V_G-\{v_1, v_3\}$ is independent and $G$ is a subgraph of the complete bipartite graph $K_{2,n}$, where $n=|V_G-\{v_1, v_3\}|$. \end{proof}

\end{document}